\documentclass[12pt]{amsart}    

\usepackage{amsmath,amssymb,latexsym} 
\usepackage{graphicx}
\usepackage{fullpage}
\usepackage[square, numbers]{natbib}   
\usepackage{frenchineq}
\usepackage{comment}
\usepackage{microtype}
\usepackage{soul}
\usepackage[colorlinks,linkcolor=blue,citecolor=magenta,urlcolor=blue]{hyperref}

\newtheorem{theorem}{Theorem}
\newtheorem{proposition}{Proposition}

\begin{document}

\title{A business dinner problem}

\author{Alejandra Estanislao}
\address[A. Estanislao]{\url{https://www.linkedin.com/in/aleestanislao} }

\author{Fr\'ed\'eric Meunier}
\address[F. Meunier]{CERMICS, \'Ecole des Ponts ParisTech, France}
\email{\href{mailto:frederic.meunier@enpc.fr}{frederic.meunier@enpc.fr}}

\begin{abstract}
We are given suppliers and customers, and a set of tables. Every evening of the forthcoming days, there will be a dinner. Each customer must eat with each supplier exactly once, but two suppliers may meet at most once at a table. The number of customers and the number of suppliers who can sit together at a table are bounded above by fixed parameters. What is the minimum number of evenings to be scheduled in order to reach this objective? This question was submitted by a firm to the Junior company of a french engineering school some years ago.
Lower and upper bounds are given in this paper, as well as proven optimal solutions with closed-form expressions for some cases.

\end{abstract}


\maketitle

\section{Introduction}

\subsection{Context}

In 2009, the following problem was submitted to the ``Junior company'' of the \'Ecole des Ponts -- one of the french engineering schools --  by a firm. We are given a set $S=\{1,\ldots,s\}$ of suppliers, a set $C=\{1,\ldots,c\}$ of customers, and a room with $t$ tables. Each evening of the forthcoming days, there will be a dinner. During a dinner, suppliers and customers sit at tables, in such a way that there are at most $\sigma\geq 1$ suppliers and at most $\gamma\geq 1$ customers sitting at the same table. Two suppliers can sit at most once at the same table. There is no similar restriction for two customers: two customers can sit as often as they want at the same table. Each customer and each supplier must sit at the same table exactly once, but not necessarily next to each other. Find a schedule (for each dinner, describe for each table the suppliers and the customers who sit at it) satisfying these constraints and minimizing the number of dinners. Note that there is always a feasible solution since whenever a supplier and a customer have not yet sat together at a table, we can add an additional dinner to allow the missing meeting.

An example of scheduling with $t=2$, $s=5$, $c=6$, $\sigma=2$, and $\gamma=3$ is given in Table~\ref{tab:ex}. The proposed solution requires six dinners. Theorems~\ref{thm:lower} and~\ref{thm:upper} below show that the optimal solution of Table~\ref{tab:ex} consists actually in three dinners.

We call this problem the {\em Business Dinner Problem}.

\subsection{Related works}\label{subsec:literature}

To authors knowledge, the Business Dinner Problem is new and has not yet been studied. However, there are some links with other problems, mainly from the combinatorial design theory.

In the Kirkman Schoolgirl Problem (see \cite{AbGeYi06}), fifteen schoolgirls must be split into five groups of three girls each day of the week, so that no two girls are in the same group twice. This problem is known to have a solution. The condition on the girls resembles the condition on the suppliers with the ``groups'' playing the role of the tables. However, it is not clear how to build from a feasible solution of the Kirkman Schoolgirl Problem a feasible solution to the Business Dinner Problem (which should then have $t=5$ tables, $s=15$ suppliers, at most $\sigma=3$ suppliers per table, and $c/\gamma\leq 7$), since there is no counterpart to the customers in the Kirkman Schoolgirl Problem.

The Social Golfer Problem is a generalization of the Kirkman Schoolgirl Problem for other values of these parameters: given $s=\sigma s'$ golfers and a time period of $r$ days, we have to split each day the golfers into $s'$ groups of $\sigma$ golfers so that no two golfers are in the same group twice. Again, it is not clear how to build from a feasible solution of the Social Golfer Problem a feasible solution to the Business Dinner Problem (which should then have $t=s'$ tables, $s$ suppliers, at  most $\sigma$ suppliers per table, and $c/\gamma\leq r$), again because there is no counterpart to the customers. The Social Golfer Problem belongs to the more general family of Tournament Scheduling Problems; see~\cite{DiFrLaWa06}. \\

There is a combinatorial object strongly connected with our problem. A square array of size $m\times m$ is a {\em Howell design} of type $H(m,2n)$, with $n\geq 1$, provided that 
\begin{enumerate}
\item every cell is either empty or contains an unordered pair of elements ({\em symbols}) chosen from a set of size $2n$,
\item every symbol occurs exactly once in each row and each column,
\item every unordered pair of symbols occurs at most once in the array.
\end{enumerate}
A Howell design provides then a feasible schedule of the Business Dinner Problem for $\lceil c/\gamma\rceil=m$, $s=2n$, $\sigma=2$, and $t\geq n$: the rows of the array are the dinners, the columns are the customer groups (of $\gamma$ or less customers), and in each cell of the array, we find suppliers present for this dinner with this customer group at a table.

The following theorem was obtained after a series of papers from the early sixties until the eighties. Two papers~\cite{ASS84,St82} conclude this series with a complete characterization of the values for which Howell designs exist.

\begin{theorem}\label{thm:howell}
A Howell design $H(m,2n)$ exists if and only if the following two conditions are satisfied:
\begin{enumerate}
\item $n\leq m\leq 2n-1$
\item $(m,2n)\notin\{(2,4),(3,4),(5,6),(5,8)\}$
\end{enumerate}
\end{theorem}

In Section~\ref{subsec:howell}, Howell designs will be used to derive upper bounds for our problem. \\

There is also a family of problems that seem closely related to ours, at least by the original formulation. It is the Oberwolfach Problem and its relatives. The Oberwolfach Problem~\cite{Gu71} first posed by Ringel in 1967 asks whether it is possible to seat an odd number $n$ of mathematicians at $m$ round tables of sizes $(\ell_1,\ldots,\ell_m)$ with $\sum_{i=1}^m\ell_i = n$ in $(n-1)/2$ dinners so that each mathematician sits next to everyone else exactly once. Equivalently, this problem consists in decomposing the complete graph $K_n$ into pairwise edge-disjoint copies of a given $2$-regular graph. The conjecture is that except for some exceptional values of the $\ell_i$'s, such a decomposition always exists; see~\cite{BrRo07} or~\cite{Tr13} for the last results on this problem. A maybe more closely related problem is the Bipartite Analogue of the Oberwolfach Problem which consists in decomposing the bipartite complete graph $K_{n,n}$ into pairwise edge-disjoint copies of a given $2$-regular graph. The complete answer was given by Piotrowski in 1991~\cite{Pi91}. It models the case when there are two delegations, each mathematician must sit between mathematicians of the other delegation, and each mathematician sits next to every mathematician of the other delegation exactly once. Generalizations of this problem for complete $n$-partite graphs are also studied; see~\cite{Li00} and~\cite{Li03}.  Another variation of the Bipartite Analogue is given by the Oberwolfach Rectangular Table Negotiation Problem for which the tables are rectangular, each delegation sits along one of the long sides of the tables, and each mathematician sits near to each mathematician of the other delegation exactly once, where ``near'' means the opposite position, the position to the right of the opposite one, or the position to the left of the opposite one. The Oberwolfach Rectangular Table Negotiation Problem has been studied in~\cite{ChMu12,Fr09,FrKu11}.

Despite the close original formulation, the Oberwolfach Problem and its variations do not seem to provide feasible solutions for the Business Dinner Problem. There is indeed a notion of neighbors at tables that is not present in the Business Dinner Problem. Moreover, and it holds also for the other problems cited above, the Business Dinner Problem is an {\bf optimization} problem, for which there are feasible solutions, whereas the problems above deal generally with the {\bf existence} of some object. This difference is also one of the motivations of our work.

\begin{table}
\begin{center}
\begin{tabular}{|l|ll|ll|} 
\hline
dinner $1$: & table $1$: &
\begin{tabular}{ll} suppliers: & 1,2\\
customers: & 1,2,3
\end{tabular} & table $2$: &
\begin{tabular}{ll} suppliers: & 3,4\\
customers: & 4,5
\end{tabular} \\
\hline

dinner $2$: & table $1$: &
\begin{tabular}{ll} suppliers: & 3\\
customers: & 2
\end{tabular} & table $2$: &
\begin{tabular}{ll} suppliers: & 5\\
customers: & 5
\end{tabular} \\
\hline

dinner $3$: & table $1$: &
\begin{tabular}{ll} suppliers: & 2\\
customers: & 4,5
\end{tabular} & table $2$: &
\begin{tabular}{ll} suppliers: & 4\\
customers: & 2,6
\end{tabular} \\ \hline

dinner $4$: & table $1$: &
\begin{tabular}{ll} suppliers: & 3\\
customers: & 1,3
\end{tabular} & table $2$: &
\begin{tabular}{ll} suppliers: & 5\\
customers: & 2,6
\end{tabular} \\ \hline

dinner $5$: & table $1$: &
\begin{tabular}{ll} suppliers: & 1,5\\
customers: & 4
\end{tabular} & table $2$: &
\begin{tabular}{ll} suppliers: & 2,3\\
customers: & 6
\end{tabular} \\ \hline

dinner $6$: & table $1$: &
\begin{tabular}{ll} suppliers: & 4,5\\
customers: & 1,3
\end{tabular} & table $2$: &
\begin{tabular}{ll} suppliers: & 1\\
customers: & 5,6
\end{tabular} \\ \hline

\end{tabular}
\end{center}
\caption{A feasible schedule in six dinners for an instance with $t=2$ tables, $s=5$ suppliers, $c=6$ customers, at most $\sigma=2$ suppliers per table, and at most $\gamma=3$ customers per table.}
\label{tab:ex}
\end{table}

\subsection{Main results}

Note first that if $c\leq\gamma$, the optimal solution is easy; see Section~\ref{subsec:cleqgamma}. 

As it will become clear in the sequel, describing an optimal solution in the general case is a very difficult task since it contains as special cases open questions of the theory of combinatorial designs. However, we are able to provide the following non-trivial lower bounds and feasible solutions in the general case. 

\begin{theorem}\label{thm:lower}
Consider the Business Dinner Problem with $t$ tables, $s$ suppliers, $c$ customers, at most $\sigma$ suppliers, and at most $\gamma$ customers simultaneously at a table. The following expressions are lower bounds when $\gamma<c$.
$$
\begin{array}{l}
\displaystyle{lb_1=\left\lceil\frac{s}{\sigma}\right\rceil,\qquad lb_2=\left\lceil\frac{c}{\gamma}\right\rceil,\qquad lb_3=\left\lceil\frac{s}{t\sigma}\left\lceil\frac{c}{\gamma}\right\rceil\right\rceil}, \\
\displaystyle{lb_4=\left\lceil\frac{\sqrt{s}}{t\gamma}\left((c-\gamma)\max\left(\sqrt{\frac{\gamma}{c-\gamma}},1\right)+\frac{\gamma}{\max\left(\sqrt{\frac{\gamma}{c-\gamma}},1\right)}\right)\right\rceil,}\\
\displaystyle{lb_5=\max_{j\in\{2,\ldots,\sigma\}}\left\lceil\frac{s}{t} \left(\frac 2 {j} \left\lceil\frac{c}{\gamma}\right\rceil-\frac{s-1}{j(j-1)}\right)\right\rceil}.
\end{array}
$$
\end{theorem}

The lower bound $lb_5$ can be made more explicit by a simple function study. The maximum of the expression over $j$ is obtained for $j=j^*$ or $j=j^*+1$ with
$$j^*=\left\lfloor\frac{1}{1-\sqrt{\frac{s-1}{s-1+2\lceil c/\gamma\rceil}}}\right\rfloor.$$ If $j^*\in[2,\sigma]$, $lb_5$ is obtained for $j=j^*$. If $j^*<2$, $lb_5$ is obtained for $j=2$. If $j^*>\sigma$, $lb_5$ is obtained for $j=\sigma$.

None of these lower bounds is strictly better than the others. The following table gives explicit values for parameters of the problem for which each of the lower bounds strictly dominates the others (indicated with a star $^*$). 

$$\begin{array}{c|ccccc}
(t,s,c,\sigma,\gamma) & lb_1 & lb_2 & lb_3 & lb_4 & lb_5 \\
\hline
(5,8,8,1,2) & 8^* & 4 & 7 & 3 & 0 \\
(6,8,8,2,1) & 4 & 8^* & 6 & 4 & 6 \\ 
(1,8,8,1,1) & 8 & 8 & 64^* & 23 & 0 \\
(1,11,8,6,4) & 2 & 2 & 4 & 7^* & 4 \\
(1,8,11,2,1) & 4 & 11 & 44 & 32 & 60^* 
\end{array}$$

\begin{theorem}\label{thm:upper}
Consider the Business Dinner Problem with $t$ tables, $s$ suppliers, $c$ customers, at most $\sigma$ suppliers, and at most $\gamma$ customers simultaneously at a table. The following expression is an upper bound when $(\lceil c/\gamma\rceil,s) \neq (2,4)$:
$$ub_1=\left\lceil\frac 2 {\sigma} \right\rceil\left\lceil\frac 1 t \min\left(\left\lceil\frac c \gamma\right\rceil,s\right)\right\rceil\max\left(\left\lceil\frac c {\gamma}\right\rceil,\left\lceil \frac s  2\right\rceil\right) \, .$$
If $(\lceil c/\gamma\rceil,s) = (2,4)$, then this upper bound is (the last term becomes a $3$)
$$ub_1 =  3 \left\lceil\frac 2 {\sigma} \right\rceil\left\lceil\frac 1 t \min\left(\left\lceil\frac c \gamma\right\rceil,s\right)\right\rceil \, .$$
If $\lceil s/\sigma\rceil\leq\lceil c/\gamma\rceil$, the following expression is also an upper bound
$$ub_2=\left\lceil\frac{1}{t}\left\lceil\frac{s}{\sigma}\right\rceil\right\rceil\left(1-\sigma+\sigma\max\left(\left\lceil\frac{c}{\gamma}\right\rceil,2\left\lceil\frac{s}{\sigma}\right\rceil\right)\right) \, .$$
Moreover, there are explicit solutions matching these upper bounds.
\end{theorem}

None of the $ub_1$ and $ub_2$ is strictly better than the other. 
If one takes  $t=3$, $s=6$, $c=3$, $\sigma=2$, and $\gamma=1$, then $ub_1$ takes the value $3$ and $ub_2$ takes the value $11$. 
If one takes  $t=3$, $s=6$, $c=9$, $\sigma=2$, and $\gamma=1$ (only the number of customers changes), then $ub_1$ takes the value $18$ and $ub_2$ takes the value $17$.

The upper bound $ub_2$ is also applicable when $\lceil s/\sigma\rceil>\lceil c/\gamma\rceil$ by making groups of at most $\sigma\lceil c/\gamma\rceil$ suppliers each; see Section~\ref{sec:upper} for details.

\subsection{Plan}

Section~\ref{sec:lower} is devoted to the proof of Theorem~\ref{thm:lower}. Optimal solutions for special cases---whose optimality is proven with the help of Theorem~\ref{thm:lower}---are described in Section~\ref{sec:spec}. These optimal solutions can be used to build feasible solutions for other values of the parameters via slight changes (Section~\ref{sec:upper}). These feasible solutions provide upper bounds for the Business Dinner Problem and prove Theorem~\ref{thm:upper}. In the last section (Section~\ref{sec:open}), some open questions are stated.

\subsection*{Acknowledgments} The authors thank Hampus Adolfsson and Peipei Han for spotting the following mistakes, which are present in the published version of this paper ({\em Journal of Combinatorial Mathematics and Combinatorial Computing}, 2016): the case $(\lceil c/\gamma\rceil,s) = (2,4)$ was missing in Theorem~\ref{thm:upper} (noted by Hampus Adolfsson); the case $s=2t-1$ was missing in Proposition~\ref{prop:cas-par} (noted by Peipei Han).

\section{Lower bounds}\label{sec:lower}

The purpose of this section is to prove Theorem~\ref{thm:lower}. It will be done by proving three propositions (Propositions~\ref{prop:straight},~\ref{prop:counting}, and~\ref{prop:boundlin}), each of them providing some of the lower bounds.

\subsection{Straightforward lower bounds}

Take a customer. During a dinner, he sits at a table with at most $\sigma$ suppliers. To meet all suppliers, he needs at least $\lceil s/\sigma\rceil$ dinners.

Take a supplier. During a dinner, he sits at a table with at most $\gamma$ customers. To meet all customers, he needs at least $\lceil c/\gamma\rceil$ dinners.

This short discussion implies

\begin{proposition}\label{prop:straight}
$\lceil s/\sigma\rceil$ and $\lceil c/\gamma\rceil$ are lower bounds for the Business Dinner Problem.
\end{proposition}

It settles the case of $lb_1$ and $lb_2$.

\subsection{A counting argument}

\begin{proposition}\label{prop:counting}
If $c>\gamma$, the quantity
$\left\lceil\frac{\sqrt{s}}{t\gamma}\left((c-\gamma)\max\left(\sqrt{\frac{\gamma}{c-\gamma}},1\right)+\frac{\gamma}{\max\left(\sqrt{\frac{\gamma}{c-\gamma}},1\right)}\right)\right\rceil$ is a lower bound for the Business Dinner Problem.
\end{proposition}

\begin{proof}
Denote $y_k$ the number of dinners for which customer $k=1,\ldots,c$ is present. Let $z$ be the maximum number of suppliers sitting simultaneously at a table among all  dinners and let $x$ be the number of customers present at that table.

On the one hand, each customer is present for at least $s/z$ dinners. Thus $\sum_{k=1}^cy_k\geq cs/z$.
On the other hand, at least $c-x$ customers are present for $z$ dinners, in order to be able to meet each of the $z$ suppliers having sat once together. Thus $\sum_{k=1}^cy_k\geq (c-x)z+xs/z$. Therefore, we have a lower bound for $\sum_{k=1}^cy_k$ being $\min_{z\in\mathbb{R}_+}\ell(z)$ with $\ell(z):=\max(cs/z,\min_{1\leq x\leq\gamma}(c-x)z+xs/z)$. 

If $z\leq\sqrt{s}$, then $cs/z\geq\min_{1\leq x\leq\gamma}(c-x)z+xs/z$ and $\ell(z)=cs/z$. If $z\geq\sqrt{s}$, then $cs/z\leq\min_{1\leq x\leq\gamma}(c-x)z+xs/z$. In this latter case,
$(c-x)z+sx/z$ is minimum for $x=\gamma$ and $\ell(z)=(c-\gamma)z+\gamma s/z$. The map $g:u\mapsto (c-\gamma)u+\gamma s/u$ attains its minimum for $u=\sqrt{\frac{s\gamma}{c-\gamma}}$. Therefore, according to the respective positions of $\sqrt{c-\gamma}$ and $\sqrt{\gamma}$, the minimum of $\ell(z)$ is attained for $z=\sqrt{s}$ or $z=\sqrt{\frac{s\gamma}{c-\gamma}}$. Thus, $$\sum_{k=1}^cy_k\geq(c-\gamma)\max\left(\sqrt{\frac{s\gamma}{c-\gamma}},\sqrt{s}\right)+\frac{s\gamma}{\max\left(\sqrt{\frac{s\gamma}{c-\gamma}},\sqrt{s}\right)} \, .$$

Now, the conclusion follows from the fact that there are at most $t\gamma$ customers present at each dinner.
\end{proof}

It settles the case of $lb_4$.

\subsubsection*{Remark} The lower bound $lb_4$ can be improved by taking into account the inequality $z\leq\sigma$ in the proof above. For sake of simplicity, we have not computed $lb_4$ with this additional constraint.

\subsection{Lower bounds through linear programming and duality}

Another way to get lower bounds consists in introducing variables $x_{i,j}\in\mathbb{Z}_+$ which count for a given schedule the number of times supplier $i$ sits at a table of exactly $j$ suppliers (him included). 
We have the following relation $\sum_{j=1}^{\sigma}\gamma x_{i,j}\geq c$, which comes from the fact that each customer sits with supplier $i$ once.
Another relation is $\sum_{j=1}^{\sigma}(j-1)x_{i,j}\leq s-1$, since two suppliers may sit at a common table at most once.
$\sum_{i\in S}\sum_{j=1}^{\sigma}\frac{1}{j}x_{i,j}$ is the number of tables needed, counted with multiplicity (each table is counted as many times it is used over the whole schedule). A lower bound is therefore given by the following linear program: 

\begin{equation}\label{eq:lb1}
\begin{array}{rlclr}\min & \displaystyle{\frac{1}{t}\sum_{i\in S}\sum_{j=1}^{\sigma}\frac{1}{j}x_{i,j}} & & & \\
\\ \mbox{s.t.} & \displaystyle{\sum_{j=1}^{\sigma}x_{i,j}} & \geq & \left\lceil\frac{c}{\gamma}\right\rceil  & i\in S
\\ & \displaystyle{\sum_{j=1}^{\sigma}(j-1)x_{i,j}} & \leq & s-1 & i\in S
\\ & x_{i,j} \in \mathbb{R}_+ & & & i\in S,\,j\in\{1,\ldots,\sigma\} \,.
\end{array}
\end{equation}

\begin{proposition}\label{prop:lin}
The linear program~\eqref{eq:lb1} has an optimal value equal to 
$$\frac{s}{t}\max\left(\frac{1}{\sigma}\left\lceil\frac{c}{\gamma}\right\rceil,\max_{j\in\{2,\ldots,\sigma\}}\left(\frac 2 j \left\lceil\frac{c}{\gamma}\right\rceil-\frac{s-1}{j(j-1)}\right)\right)\, .$$
\end{proposition}

\begin{proof}
Program~\eqref{eq:lb1} is separable in $i$. Its study reduces therefore to
\begin{equation}\label{eq:lb2}
\begin{array}{rlclr}\min & \displaystyle{\sum_{j=1}^{\sigma}\frac{1}{j}x_{j}} & & & \\
\\ \mbox{s.t.} & \displaystyle{\sum_{j=1}^{\sigma}x_{j}} & \geq & \left\lceil\frac{c}{\gamma}\right\rceil  & 
\\ & \displaystyle{\sum_{j=1}^{\sigma}(j-1)x_{j}} & \leq & s-1 & 
\\ & x_{j} \in \mathbb{R}_+ & & & j\in\{1,\ldots,\sigma\} \, .
\end{array}
\end{equation}

The dual of program~\eqref{eq:lb2} is
\begin{equation}\label{eq:duallb2}
\begin{array}{rlclr}\max & \left\lceil\frac{c}{\gamma}\right\rceil\lambda-(s-1)\mu & & & 
\\ \mbox{s.t.} & \lambda\leq (j-1)\mu+\frac{1}{j} & & \quad j\in\{1,\ldots,\sigma\} \\
 & \lambda,\mu\in\mathbb{R}_+\, . & 
\end{array}
\end{equation}

Let $(\lambda^*,\mu^*)$ be an optimal solution of~\eqref{eq:duallb2}. We have $$\lambda^*=\min_{j\in\{1,\ldots,\sigma\}}(j-1)\mu^*+\frac{1}{j}$$ and thus an optimal value equal to 
\begin{equation}\label{eq:maxmin}\max_{\mu\geq 0}\min_{j\in\{1,\ldots,\sigma\}}\left((j-1)\left\lceil\frac{c}{\gamma}\right\rceil-(s-1)\right)\mu+\frac{1}{j}\left\lceil\frac{c}{\gamma}\right\rceil \,.
\end{equation}

To evaluate this latter expression, we write $$[0,+\infty)=\left[0,\frac{1}{\sigma(\sigma-1)}\right]\cup\left(\bigcup_{k=2}^{\sigma-1}\left[\frac{1}{(k+1)k},\frac{1}{k(k-1)}\right]\right)\cup[1/2,+\infty),$$ and we optimize for $\mu$ belonging to each of these intervals. 

Remark that the map $g:j\mapsto\mu j+\frac{1}{j}$ is decreasing for $j\leq\frac{1}{\sqrt{\mu}}$ and increasing otherwise.

\begin{description}
\item[Case $\mu\in{[0,\frac{1}{\sigma(\sigma-1)}]}$] According to the remark concerning the map $g$, we have $$\min_{j\in\{1,\ldots,\sigma\}}\left((j-1)\left\lceil\frac{c}{\gamma}\right\rceil-(s-1)\right)\mu+\frac{1}{j}\left\lceil\frac{c}{\gamma}\right\rceil=\min_{j\in\{\sigma-1,\sigma\}}\left((j-1)\left\lceil\frac{c}{\gamma}\right\rceil-(s-1)\right)\mu+\frac{1}{j}\left\lceil\frac{c}{\gamma}\right\rceil.$$

Comparing the value obtained for $j=\sigma-1$ and $j=\sigma$ leads to the maximum on this interval, depending on the sign of $(\sigma-1)\left\lceil\frac{c}{\gamma}\right\rceil-(s-1)$,
$$\max\left(\frac{1}{\sigma}\left\lceil\frac{c}{\gamma}\right\rceil,\frac 2 {\sigma}\left\lceil\frac{c}{\gamma}\right\rceil-\frac{s-1}{(\sigma-1)\sigma}\right).$$\\

\item[Case $\mu\in{[\frac{1}{(k+1)k},\frac{1}{k(k-1)}]}$] 
For such a $\mu$, according to the remark on $g$, the $j$ realizing the minimum is in $\{k-1,k,k+1\}$. A straightforward computation shows that the minimum is actually reached on $k$.

\begin{itemize}
\item If $\lceil c/\gamma\rceil(k-1)\leq s-1$, then the maximum is reached for $\mu=\frac{1}{(k+1)k}$ and we get a maximum equal to $$\frac 2 {k+1}\left\lceil\frac{c}{\gamma}\right\rceil-\frac{s-1}{(k+1)k}.$$
\item If $\lceil c/\gamma\rceil(k-1)\geq s-1$, then the maximum is reached for $\mu=\frac{1}{k(k-1)}$ and we get a maximum equal to $$\frac 2 k\left\lceil\frac{c}{\gamma}\right\rceil-\frac{s-1}{k(k-1)}.$$\\
\end{itemize}

\item[Case $\mu\in{[1/2,+\infty)}$] Exchanging $\min$ and $\max$ in Equation~\eqref{eq:maxmin} increases the expression. Letting $j=1$ once we have exchanged leads therefore to an upper bound, which is equal to $\lceil c/\gamma\rceil-(s-1)/2$. This quantity is also a lower bound because it is what we get when we evaluate Equation~\eqref{eq:maxmin} for $\mu=1/2$.
\end{description}

Therefore, by strong duality, the optimal value of~\eqref{eq:duallb2} is 
\begin{equation}\label{eq:dualfin}
\max\left(\frac{1}{\sigma}\left\lceil\frac{c}{\gamma}\right\rceil,\max_{j\in\{2,\ldots,\sigma\}}\left(\frac 2 j\left\lceil\frac{c}{\gamma}\right\rceil-\frac{s-1}{j(j-1)}\right)\right).\end{equation}
\end{proof}

To get the optimal $j$ as stated just after Theorem~\ref{thm:lower}, we compute the values $x\geq 1$ for which $x\mapsto\frac 2 x \left\lceil\frac{c}{\gamma}\right\rceil-\frac{s-1}{x(x-1)}$ is an increasing map. A simple calculation leads to
$$x\in\left[1;\frac{1}{1-\sqrt{\frac{s-1}{s-1+2\lceil c/\gamma\rceil}}}\right]\, .$$ The integer $j^*$ maximizing the expression~\eqref{eq:dualfin} is therefore one of the two integers around the right bound of this interval.

Proposition~\ref{prop:lin} and the preceding discussion lead to the following proposition.

\begin{proposition}\label{prop:boundlin}
$$\left\lceil\frac{s}{t}\max\left(\frac{1}{\sigma}\left\lceil\frac{c}{\gamma}\right\rceil,\max_{j\in\{2,\ldots,\sigma\}}\left(\frac 2 j \left\lceil\frac{c}{\gamma}\right\rceil-\frac{s-1}{j(j-1)}\right)\right)\right\rceil$$ is a lower bound for the Business Dinner Problem.
\end{proposition}

It gives $lb_3$ and $lb_5$.

\section{Special cases with optimal solutions}\label{sec:spec}

\subsection{Case $c\leq\gamma$}\label{subsec:cleqgamma} In this case, the optimal solution is $\lceil s/\sigma\rceil$. Indeed, all customers sit together at a table, and each evening, they have a dinner with exactly $\sigma$ suppliers, except maybe the last evening, when they have a dinner with $s-\sigma \lfloor s/\sigma\rfloor$ suppliers. It gives $\lceil s/\sigma\rceil$ dinners. Since it is also a lower bound (Proposition~\ref{prop:straight}), it is an optimal solution.

\subsection{Case $\sigma=1$} We have the following proposition.
\begin{proposition}\label{prop:sigma1}
If $\sigma=1$, then the optimal number of dinners is $$\max\left(s,\left\lceil\frac c \gamma\right\rceil,\left\lceil\frac s t\left\lceil \frac c \gamma\right\rceil\right\rceil\right)\, .$$
\end{proposition}

\begin{proof}
Assume $\gamma=1$. Consider the complete bipartite graph $K_{s,c}$ with on one side the suppliers and on the other side the customers. If the dinners are the colors, we want to find a proper edge-coloring of $K_{s,c}$ with each color being present at most $t$ times. According to a theorem by De Werra \cite{DeW70}, the minimal number of colors in a proper edge-coloring of a bipartite graph $G=(V,E)$ with each color being present at most $t$ times is $\max(\lceil |E|/t\rceil,\Delta(G))$, where $\Delta(G)$ is the maximal degree of $G$. Therefore, the minimal number of colors in $K_{s,c}$ is $\max(\lceil sc/t\rceil,s,c)$.

The case $\gamma>1$ is obtained as follows: split the customers in $\lceil c/\gamma\rceil$ groups, each being of size $\leq\gamma$; according to the preceding construction, we get a feasible solution with $\max(s,\lceil c/\gamma\rceil,\lceil s/t\lceil c/\gamma\rceil\rceil)$ dinners. This solution is optimal according to the lower bounds $lb_1$, $lb_2$, and $lb_3$ of Theorem~\ref{thm:lower}.
\end{proof}

\subsection{Case $\sigma=2$, $s>c/\gamma$ and $t\geq\min(c/\gamma,s/2)$} \label{subsec:howell}

\begin{proposition}\label{prop:dinHow}
If  $\sigma=2$, $s>c/\gamma$, $t\geq\min(c/\gamma,s/2)$ and $(\lceil c/\gamma\rceil,s)\neq (2,4)$, the optimal number of dinners is $$\max\left(\left\lceil\frac c \gamma\right\rceil,\left\lceil\frac s 2 \right\rceil\right)\, .$$ 
\end{proposition}

This case coincides more or less with the concept of Howell designs, presented in Section~\ref{subsec:literature}. With the help of Theorem~\ref{thm:howell}, we are able to prove Proposition~\ref{prop:dinHow}.

\begin{proof}[Proof of Proposition~\ref{prop:dinHow}]
Suppose first that $(\lceil c/\gamma\rceil,s)\notin\{(3,4),(5,6),(5,8)\}$. Because of the assumption of the theorem, we have also $(\lceil c/\gamma\rceil,s)\neq(2,4)$. Therefore, we can apply Theorem~\ref{thm:howell}.
\begin{description}
\item[Case $\lceil c/\gamma\rceil\geq s/2$]
If $s$ is even, let $2n=s$. With the interpretation of the Howell design $H(\lceil c/\gamma\rceil,2n)$ in terms of schedule, we are able to find a solution in $\lceil c/\gamma\rceil$ dinners, which is obviously optimal as it is also the lower bound $lb_2$.

If $s$ is odd, let $2n-1=s$, add a fictitious supplier and we get again an optimal solution in $\lceil c/\gamma\rceil$ evenings with the same construction.
\item[Case $s/2>\lceil c/\gamma\rceil$]
As in the case above, we define $n$ such that $s=2n$ or $s=2n-1$ according to the parity of $s$. With $H(n,2n)$ from which we keep only the $\lceil c/\gamma\rceil$ first columns, we get a feasible schedule in $n$ dinners. As $n=\lceil s/\sigma\rceil$, which is the lower bound $lb_1$, we get the optimality of this solution.
\end{description}

\begin{table}[ht!]
\centering
\begin{tabular}{|c|c|c|}\hline 1,2 & 3,4 & \\ \hline 3 & 1 & 2,4 \\ \hline 4 & 2 & 1,3 \\ \hline \end{tabular}
\caption{An optimal schedule for  $t=3$ tables, $s=4$ suppliers, $\lceil c/\gamma\rceil=3$ groups of customers (the columns), at most $\sigma=2$ suppliers per table, in $3$ dinners (the rows).}
\label{tab:s4c3}
\end{table}

\begin{table}[ht!]
\centering
\begin{tabular}{|c|c|c|c|c|}
\hline  & 6 & 2,3 & 4,5 & 1\\ 
\hline 6 & 3,4 & 1 & 2 & 5 \\ 
\hline 3,5 & 1,2 &  & 6 & 4\\ 
\hline 2,4 &  & 5 & 1 & 3,6 \\ 
\hline 1 & 5 & 4,6 & 3 & 2 \\ \hline \end{tabular}
\caption{An optimal schedule for $t=5$ tables, $s=6$ suppliers, $\lceil c/\gamma\rceil=5$ groups of customers (the columns), at most $\sigma=2$ suppliers per table,  in $5$ dinners (the rows).}
\label{tab:s6c5}
\end{table}

\begin{table}[ht!]
\centering
\begin{tabular}{|c|c|c|c|c|}\hline 4 & 6 & 1,5 & 7,8 & 2,3\\ \hline 2,6 & 5,7 & 3 & 1 & 4,8 \\ \hline 7 & 1,3 & 2,8 & 4,5 & 6\\ \hline 1,8 & 2,4 & 7 & 3 & 5 \\ \hline 3,5 & 8 & 4,6 & 2 & 1,7 \\ \hline \end{tabular}
\caption{An optimal schedule for $t=5$ tables, $s=8$ suppliers, $\lceil c/\gamma\rceil=5$ groups of customers (the columns), at most $\sigma=2$ suppliers per table, in $5$ dinners (the rows).}
\label{tab:s8c5}
\end{table}

For the three remaining cases $(\lceil c/\gamma\rceil,s)\in\{(3,4),(5,6),(5,8)\}$, we use the explicit solutions\footnote{computed with the help of the constraint programming toolbox of the \texttt{IBM ILOG CPLEX Optimization Studio}.} of Tables~\ref{tab:s4c3}--\ref{tab:s8c5}. The optimality is proven with the help of lower bound $lb_2$.
\end{proof}

For the case $(\lceil c/\gamma\rceil,s)=(2,4)$, we have a solution in three dinners---see in Table~\ref{tab:s4c2}---and it is easy to see that there is no solution in two dinners.

\begin{table}
\centering
\begin{tabular}{|c|c|}\hline 1,2 & 3,4 \\ \hline 3 & 1 \\ \hline 4 & 2 \\ \hline \end{tabular}
\caption{An optimal schedule for $t=2$ tables, $s=4$ suppliers, $\lceil c/\gamma\rceil=2$ groups of customers (the columns), at most $\sigma=2$ suppliers per table, in $3$ dinners (the rows).}
\label{tab:s4c2}
\end{table}

\subsubsection*{Remark} It becomes clear that a general exact solution for the present problem is out of reach: such a solution would describe under which conditions objects generalizing Howell designs with more than two numbers in each cell would exist---a difficult topic which is still under investigation.

\subsection{Case $t=\lceil s/2\rceil$, $\sigma=2$, and $\lceil c/\gamma\rceil\geq \frac{3}{2}s$}

\begin{proposition}\label{prop:cas-par}
Suppose $t=\lceil s/2\rceil$, $\sigma=2$, and $\lceil c/\gamma\rceil\geq\frac{3}{2}s$.

If $s=2t$, then the optimal number of dinners is $2\left\lceil\frac c \gamma\right\rceil-s+1$.

If $s=2t-1$, then the optimal number of dinners is $\left \lceil \left\lceil\frac c \gamma\right\rceil \frac {s} t  -\frac{1}{t} \right\rceil  +3-2t$.
\end{proposition}

\begin{proof}
We first write the proof for the case $s=2t$. 

We prove that there is a feasible solution matching this number of dinners.

Let $c'=\lceil c/\gamma\rceil$. We split the whole set of customers into $c'$ groups, each of them with at most $\gamma$ customers. Consider an optimal schedule with $s$ suppliers and $s-1$ groups of customers and derive from it a feasible schedule in $s-1$ dinners (Proposition~\ref{prop:dinHow}). At the end of these $s-1$ dinners, we have $s-1$ groups among the $c'$ groups that have met all suppliers. For the remaining $c'-s+1$ groups, we use a feasible solution in $\lceil\frac s t (c'-s+1)\rceil=2c'-2s+2$ dinners, given by Proposition~\ref{prop:sigma1}. The suppliers being alone in such a solution, we get a feasible solution for the whole collection of $c'$ groups in $2c'-s+1$ dinners. 

This number of dinners is optimal according to the lower bound $lb_5$ of Theorem~\ref{thm:lower} with $j=2$.\\

Let us deal with the case $s=2t-1$. Adding a fictitious supplier, we have $s'=s+1$ suppliers and we get similarly as above a solution in $s'-1+\lceil\frac s t (c'-s'+1)\rceil$ dinners. Note that the $c'-s'+1$ groups of customers have $s$ suppliers to meet and we get a $\frac s t$ term and not a $\frac{s'} {t}$ one. Replacing $s'$ and $s$ by their expressions in $t$, we get a solution in $\lceil\frac{c's} t +3-\frac{1}{t}-2t\rceil$ dinners.

The lower bound $lb_5$ of Theorem~\ref{thm:lower} with $j=2$ is in this case $\lceil\frac {c's} t -\frac{4t^2-6t+2}{2t}\rceil=\lceil\frac {c's} t +3-\frac{1}{t}-2t\rceil$ and proves the optimality of this solution.
\end{proof}

\subsection{Case $t=\gamma=1$, $c\leq p\leq\sigma$, $s=p^2$ with $p$ prime}\label{subsec:t1} 

\begin{proposition}
 If $t=\gamma=1$, $c\leq p\leq\sigma$, $s=p^2$ with $p$ prime, the optimal number of dinners is $pc$. 
\end{proposition}

\begin{proof}
We define the following $p\times p$ matrices $M^{(1)},\ldots, M^{(c)}$ where $$M_{i,j}^{(k)}=j+p(kj-j-k+i)\mbox{ mod }p^2\quad\mbox{for $i,j\in\{1,\ldots,p\}$ and for $k\in\{1,\ldots,c\}$.}$$ For a $k\in\{1,\ldots,c\}$, let $i,j,i',j'$ be in $\{1,\ldots,p\}$. If we have $M_{i,j}^{(k)}=M_{i',j'}^{(k)}\mbox{ mod }p^2$, then we get first that $j=j'$ (by counting modulo $p$), then that $i=i'$. It means that the numbers appearing in a $M^{(k)}$ are all distinct modulo $p^2$, and hence that all the numbers from $1$ to $1+p^2$ modulo $p^2$ appear in such a matrix since there are $p^2$ entries. 

We build now a feasible schedule in $pc$ dinners as follows. For each $k$, the matrix $M^{(k)}$ encodes $p$ dinners with customer $k$ at the unique table: each row provides the suppliers present at a dinner. According to the remark above, each customer meets all suppliers. It remains to check that two suppliers eat at most once together. Assume that it is not the case. Because all numbers in a $M^{(k)}$ are distinct, we know that two suppliers eat at most once together, when one considers the dinners with a given customer. Therefore, if two suppliers eat at least twice together we have simultaneously $M_{i,j_1}^{(k)}=M_{i',j'_1}^{(k')}$ and $M_{i,j_2}^{(k)}=M_{i',j'_2}^{(k')}$ (counted modulo $p^2$) for some $i,i',j_1,j'_1,j_2,j'_2\in\{1,\ldots,p\}$ with $j_1\neq j_2$ and $j'_1\neq j'_2$, and distinct $k$ and $k'$ in $\{1,\ldots,c\}$. We get first that $j_1=j'_1$ and that $j_2=j_2'$ (by counting modulo $p$). Then we get that $p^2$ divides $p(i-i'+kj_1-k'j_1+k'-k)$ and $p(i-i'+kj_2-k'j_2+k'-k)$, which means that $p$ divides $(k-k')(j_1-j_2)$. Since $p$ is prime, $|k-k'|<p$ and $|j_1-j_2|<p$, we get a contradiction.

The optimality of the solution is clear since it matches $lb_4$ of Theorem~\ref{thm:lower}.
\end{proof}

\section{Feasible solutions}\label{sec:upper}

The main purpose of this section is to prove Theorem~\ref{thm:upper}.

\subsection{Domination results}

Let us denote $r(t,s,c,\sigma,\gamma)$ the optimal solution of the problem. Obviously, it is an non-decreasing map in $s$ and in $c$.
It is also a non-increasing map in $t$, in $\sigma$, and in $\gamma$, but concerning these quantities, we can say a little bit more.

\begin{equation}\label{eq:t}
r(t_1,s,c,\sigma,\gamma)\leq \left\lceil\frac{t_2}{t_1}\right\rceil r(t_2,s,c,\sigma,\gamma)\qquad\mbox{for any $t_1$ and $t_2$}\, .
\end{equation}

Indeed, if $t_1\geq t_2$, then any solution with $t_2$ tables is also a solution with $t_1$ tables ($r$ is a non-increasing map).
And if $t_2\geq t_1$, we can split the set of tables into $\left\lceil\frac{t_2}{t_1}\right\rceil$ groups of $t_1$ tables, and use any solution of the problem with $t_2$ tables to build a solution with $t_1$ tables, each group corresponding to a distinct dinner.

With the same kind of reasoning, we get

\begin{equation}\label{eq:sigma}
r(t,s,c,\sigma_1,\gamma)\leq \left\lceil\frac{\sigma_2}{\sigma_1}\right\rceil r(t,s,c,\sigma_2,\gamma)\qquad\mbox{for any $\sigma_1$ and $\sigma_2$}\, .
\end{equation}

Making groups of at most $\gamma_1$ customers leads to
\begin{equation}\label{eq:gamma}
r(t,s,c,\sigma,\gamma_2)\leq r(t,s,\lceil c/\gamma_1\rceil,\sigma,\lceil\gamma_2/\gamma_1\rceil)\qquad\mbox{if $\gamma_1\leq\gamma_2$}\, .
\end{equation}

Finally, we have also the following relation.
\begin{equation}\label{eq:s}
r(t,s,c,\sigma,\gamma)\leq r(t,s_1,c,\sigma,\gamma)+r(t,s_2,c,\sigma,\gamma)\qquad\mbox{for any $s_1$ and $s_2$ s.t. $s_1+s_2=s$} \, .
\end{equation}

\subsection{Explicit solutions}

\begin{proof}[Proof of Theorem~\ref{thm:upper}]
We get $ub_1$ as follows. 

 Let $t_1=t$, $t_2=\min\left(\left\lceil\frac c \gamma\right\rceil,s\right)$, $\sigma_1=\sigma$ and $\sigma_2=2$.  Suppose that $(\lceil c/\gamma\rceil,s) \neq (2,4)$. We have then $r(t_2,s,c,\sigma_2,\gamma)\leq\max\left(\left\lceil\frac c {\gamma}\right\rceil,\left\lceil \frac s  2\right\rceil\right)$. Indeed, if $s>c/\gamma$, we use Proposition~\ref{prop:dinHow}; and if $s\leq c/\gamma$, we use Proposition~\ref{prop:sigma1} combined with the inequality $r(t_2,s,c,2,\gamma)\leq r(t_2,s,c,1,\gamma)$, which follows from Equation~\eqref{eq:sigma}. 
Equations~\eqref{eq:t} and~\eqref{eq:sigma} are then used to conclude.

Suppose now that $(\lceil c/\gamma\rceil,s) = (2,4)$. We have then $r(t_2,s,c,\sigma_2,\gamma) = 3$, as noted at the end of Section~\ref{subsec:howell}, and we finish as above.
\\

To get $ub_2$, we first define a new business dinner problem with $t'=\lceil s/\sigma \rceil$ tables, $s'=t'\sigma$ suppliers, $\lceil c/\gamma\rceil$ customers, at most $\sigma$ suppliers per table and at most $1$ customer per table. For this problem, a feasible solution consists in taking $t'$ customers, and in putting each of them at a separate table. The first evening, each of these $t'$ customers eats with $\sigma$ distinct suppliers. The $s'-\sigma$ following evenings, each of these $t'$ customers eats with one of the $s'-\sigma$ suppliers he has not yet eaten with. It is easy to schedule these dinners. For the $\lceil c/\gamma\rceil-t'$ remaining customers, we build a solution with exactly one supplier per table in $\sigma\max(\lceil c/\gamma\rceil-t',t')$ dinners (the schedule follows from Proposition~\ref{prop:sigma1}).
We get therefore a feasible solution in $1+s'-\sigma+\sigma\max(\lceil c/\gamma\rceil-t',t')$ dinners:
$$r(t',s',\lceil c/\gamma\rceil,\sigma,1)\leq 1-\sigma+\sigma\max(\lceil c/\gamma\rceil,2t')\, .$$

Combining this solution with the monotony of $r$ in $s$, Equation~\eqref{eq:t} and Equation~\eqref{eq:gamma} (for $\gamma_2=\gamma_1=\gamma$), we get $ub_2$.\end{proof}
${}$\\

Actually, if $s>c/\gamma$, the upper bound $ub_1$ can be improved with the same proof (Proposition~\ref{prop:dinHow} can be used in a tighter way) by 
$$\left\lceil\frac 2 {\sigma} \right\rceil\left\lceil\frac 1 t \min\left(\left\lceil\frac c \gamma\right\rceil,\left\lceil \frac s  2\right\rceil\right)\right\rceil\max\left(\left\lceil\frac c {\gamma}\right\rceil,\left\lceil \frac s  2\right\rceil\right)\, .$$\\

The upper bound $ub_2$ leads to feasible solutions even if $s/\sigma>c/\gamma$. Indeed, we can make the Euclidean division of $\lceil s/\sigma\rceil$ by $\lceil c/\gamma\rceil$. Let us write $$\lceil s/\sigma\rceil = q\lceil c/\gamma\rceil+\rho\quad\mbox{with $0\leq\rho<\lceil c/\gamma\rceil$.}$$ Making $q$ groups of at most $\sigma\lceil c/\gamma\rceil$ suppliers each and one group of at most $\sigma\rho$ suppliers, and with the help of Equation~\eqref{eq:s}, we get a solution in at most 
\begin{equation}\label{eq:eucli}q\left\lceil\frac{1}{t}\left\lceil\frac{c}{\gamma}\right\rceil\right\rceil\left(1-\sigma+2\sigma\left\lceil\frac{c}{\gamma}\right\rceil\right)+\left\lceil\frac{\rho}{t}\right\rceil\left(1-\sigma+2\sigma\max\left(\left\lceil\frac{c}{\gamma}\right\rceil,2\rho\right)\right)\end{equation} dinners.

\section{Some open questions}\label{sec:open}

\subsection{Improving the bounds}

For many instances, the ratio (best upper bound) / (best lower bound) is low. It is however possible to make this ratio arbitrarily high as follows. Set $s:=x$ and $\sigma:=\sqrt{x}$ and let $x\rightarrow +\infty$. The parameters $c$, $\gamma$, and $t$ are considered as constant. The notation $g=O(f)$ means that $g/f$ is asymptotically bounded above by a constant, whereas the notation $g=\Omega(f)$ means that $g/f$ is asymptotically bounded below by a constant.

Then $lb_2$ is constant. The lower bound $lb_5$ goes to something negative: it can be checked that $j^*=\Omega(x)$ and therefore the maximum is reached on $\sigma$ when $x$ is sufficiently high. The three others lower bounds are $O(\sqrt{x})$.

We have $ub_1=\left\lceil\frac 1 t \left\lceil\frac c \gamma\right\rceil\right\rceil\left\lceil \frac x  2\right\rceil$ when $x$ is sufficiently high. For the other upper bound, using Equation~\eqref{eq:eucli}, we get $ub_2\sim \frac{2\left\lceil\frac c \gamma\right\rceil-1}{\left\lceil\frac c \gamma\right\rceil}\left\lceil\frac 1 t \left\lceil\frac c\gamma\right\rceil\right\rceil x$.

The ratio is therefore a $\Omega(\sqrt{x})$. \\

A question is whether it is possible to improve the lower and upper bounds in order to get a ratio bounded above by a constant.

\subsection{Making groups}

The most intriguing open question is the following, as a positive answer seems intuitively correct, 
\begin{quote}
{\em Is there always an optimal solution in which the customers are split into groups, the members of each group staying together for all dinners?}
\end{quote} 
Indeed, this property is satisfied in all optimal or good solutions proposed in the present paper. An alternative formulation is whether Equation~\eqref{eq:gamma} is actually an equality when $\gamma_1=\gamma_2$. \\

\subsection{One table and at most one customer per table}

An open question which may be tractable is the case with only one table and at most one customer at each table. A partial answer is given Section~\ref{subsec:t1} and we were not able to deal with the general case.

\bibliographystyle{amsplain}
\bibliography{Dinners}

\end{document}